\documentclass[a4paper,oneside, 12pt]{article}

\usepackage{amsmath, amsthm, amsfonts, amssymb}
\usepackage[all]{xy}
\usepackage{graphicx}
\usepackage{MnSymbol}
\usepackage{tikz}
\usepackage{hyperref}
\usepackage[nottoc]{tocbibind}
\usepackage{hyperref}
\usetikzlibrary{matrix, arrows,positioning}
\usepackage{float}
\usepackage{amsmath}
\numberwithin{equation}{section}

\newtheorem{theorem}{Theorem}[section]
\newtheorem{corollary}[theorem]{Corollary}
\newtheorem{lemma}[theorem]{Lemma}
\newtheorem{proposition}[theorem]{Proposition}

\theoremstyle{definition}

\newtheorem{example}[theorem]{Example}
\newtheorem{remark}[theorem]{Remark}

\newcommand{\C}{\mathbb{C}}
\newcommand{\R}{\mathbb R}
\newcommand{\F}{\mathcal F}

\newcommand{\Q}{\mathbb Q}
\newcommand{\Z}{\mathbb Z}
\newcommand{\N}{\mathbb N}
\newcommand{\p}{\mathbb P}

\newcommand{\xto}{\xrightarrow}

\def\url#1{\noindent \sf{#1}}

\begin{document}

\title{\bf Monodromy problem and Tangential center-focus problem for product of generic lines in $\p^2$}
\author{Daniel L\'opez Garc\'ia}
\date{}
\maketitle

\begin{abstract}
	We consider the rational map $F$ defined by the quotient of products of lines in  general position and we study the monodromy problem  and tangential center-focus problem for the fibration associated with  $F$. Thus, we study the submodule of the 1-homology group of a regular fiber of $F$ generated by the orbit of the monodromy action on a vanishing cycle. Moreover, we characterize the meromorphic 1-forms $\omega$ in $\mathbb{P}^2$ such that the Abelian integral $\int_{\delta_t}\omega$ vanishes on a family of cycles $\delta_t$ around a center singularity.
\end{abstract}



\section*{Introduction}

One of the most explored approaches to the 16th Hilbert  problem  is through the study of Abelian integrals, which arise from the displacement function of a planar vector field. The analysis of the first variation of the displacement function is a way of bounding the number of limit cycles in one-parameter family of polynomial forms, perturbing a Hamiltonian form. Namely, given a system  
\begin{equation}
\label{weaksystem}
dF+\varepsilon \omega=0, \text{ where }\omega=p(x,y)dx+q(x,y)dy\text{ , }F,p,q\in \C[x,y] \text{ and }\varepsilon \text{ small enough},
\end{equation}
a periodic orbit $\delta_{z_0}\subset F^{-1}(t_0)$, and a transversal section $\Sigma$ to the solution of \eqref{weaksystem} in the point $\delta_{t_0}(0)$, we can define the displacement function  $f(t,\varepsilon):=\Delta_{\Sigma}(t)-t$, where $\Delta_{\Sigma}(t)$ is the Poincar\'e first return map. 

We can choose the transversal section $\Sigma$ parametrized by $F=t$. It is a known fact that the displacement function satisfies 
$$f(t,0)\equiv 0 \hspace{2mm}\text{, and }\hspace{2mm}\frac{d}{d \varepsilon}\bigg{|}_{\varepsilon=0}f(t,\varepsilon)=-\int_{\delta_t}\omega,$$
where $\delta_t\subset F^{-1}(t)$ is a loop for $t$ close to $t_0$, see for example \cite{Sergie}. Thus, it is defined the Abelian integral, as the complex multivalued function $I:\C\to \C\text{ given by }I_1(t)=-\int_{\delta_t}\omega$. The function $f(t,\varepsilon)$ is analytic in $(t, \varepsilon)$, hence we can write its Taylor series 
$$f(t, \varepsilon)=I_0(t)+\varepsilon I_1(t)+\varepsilon^2I_2(t)+O(\varepsilon^2),$$
where $I_0(t)\equiv 0$, and $I_1(t)$ is  the previously defined Abelian integral. Since $I_1\neq 0$, the number of  zeros of $I_1$ is an upper bound for the number of limit cycles of \eqref{weaksystem}. Bound the number of zeros of $I_1(t)$ is known as the tangential 16th Hilbert  problem, (see \cite{ChMonodromy, Sergie}). If $I_1\equiv 0$, it is necessary to study higher-order perturbations.

Then naturally arises the problem of classifying the conditions of $F$ and $\omega$ with which $I_1(t)$ vanishes over a family  $\delta_t$ of cycles around a Morse singularity or center singularity. In \cite{ChMonodromy} the authors call this problem the  \textit{tangential center-focus problem}. In \cite{Il}, Y. Ilyashenko proves  that for a generic polynomial $F$ ($F$ with isolated singularities, and  such that the critical points and critical values are different) $I_1(t)$ is null along a family of cycles around to a center if and only if $\omega=dA+BdF$ for $A,B\in \C[x,y]$. This characterization is done by showing that the monodromy action is transitive. Therefore, $I_1$ is zero over any cycle in the 1-homology group of any regular fiber of $F$, and the result is derived from the work of L. Gavrilov on Petrov modules \cite{gavrilovpetrov}.

The \textit{monodromy problem} asks for conditions on  $F$ such that the subspace generated by the orbit of the monodromy action on a vanishing cycle of a Morse singularity is equal to $H_1(F^{-1}(t), \Q)$ for any $t$ regular value of $F$. Moreover, in \cite{ChMonodromy}, C. Christopher and P. Marde{\v{s}}i{\'c} solve the monodromy problem for the hyperelliptic curves $F(x,y)=y^2+p(x)$. It is worth noting that this family of polynomials is closely related to the 0-dimensional case, which is studied in \cite{GavrilovMovasati}. In \cite{DanielLopezMonodromy4}, this problem is studied for the families of  polynomials $y^3+p(x)$ and $y^4+p(x)$, by using Dynkin diagrams.  Moreover, in this work the monodromy problem is solved for the  case of $p(x)+q(y)$ with $\deg(p)=\deg(q)=4$.

In \cite{hosseincenterlog}, H. Movasati studies the monodromy action for a product of lines in generic position  $F=l_0l_1\cdots l_d$ by using the results in \cite{ACampo}. Furthermore, he shows that if $I_1$ is zero along a family of vanishing cycles $\delta_t$, then $\omega=F\sum_{k=0}^d\lambda_k\frac{dl_k}{l_k}+dP$  where $\lambda_k\in \C$ and $P$ is a polynomial. The case of product of lines in generic position is also studied in \cite{MarcoU}. A generalization is given in \cite{pontigotangential}, where the irreducible components of $F$ are of the form $f_k^{n_k}+g_k$ where $n_k\in \N$, $f_k$ are linear polynomials without constant term, and $\deg(g_k)<n_k$. Another work where the tangential center-focus problem for non-generic polynomials is studied is \cite{yadollahcenter}, where the pullback case is studied.

In this work we study the monodromy problem and the tangential center-focus problem in the projective space $\p^2$. Some works like \cite{hosseinabelian} and \cite{yadollahBrieskorn} study these problems in the projective space as well for generic integrable foliations and pullback integrable foliations, respectively. In the projective context, it is necessary to consider rational functions $\frac{P(x,y,z)}{Q(x,y,z)}$, with $P$ and $Q$ being homogeneous polynomials of the same degree.  Here we consider the rational function defined by products of lines in generic position, that is $$F(x,y,z)=\frac{P(x,y,z)}{Q(x,y,z)}:=\frac{\Pi_{k=0}^d R_k}{\Pi_{k=d+1}^{2d+1}R_k}\text{ , where }R_k=(2d+1-k)x+ky-k(2d+1-k)z.$$  

The article is divided in two aspects, the first one is a topological study of the fibration defined by $F$. In section \ref{SectionHomology} we show that the vanishing cycles $\{\delta^P_{i}, \Delta^P_{j}, \delta^Q_{i}, \Delta^Q_{j}, \sigma_k\}$, $i=1,\ldots , \frac{d(d+1)}{2}$, $j=1,\ldots, \frac{d(d-1)}{2}$, $k=1, \ldots,d^2$, associated with the critical points of $F$ generate the 1-homology group of any regular fiber. In section \ref{seccionIntMatrix}  we compute the intersection matrix for the vanishing cycles of the fibration defined by $F$.
 Section \ref{seccionOrbits} is intended to study the monodromy problem for the rational map $F$. In this section we present the first main theorem of the article. We show that the orbit of  monodromy action on a vanishing generates the 1-homology group of any compact regular fiber. 
\begin{theorem}
For any regular value $b$ of $F$, the action of the monodromy on any vanishing cycle generates the homology $H_1(\overline{F^{-1}(b)})$.
\end{theorem}
The second aspect of the article is an application of the previous topological results to the study of holomorphic foliations in $\p^2$. The bridge between the topological aspect and the algebraic aspect of the foliations is through the Gauss-Manin connection. In section \ref{Brieskorn} we study the Brieskorn lattices/Petrov modules for 1-forms in $\p^2$ with pole along $D=\{Q=0\}$, and we also study the  Gauss-Manin connection $\nabla$ for $F$. In particular we compute the kernel of  $\nabla^2$. Finally, in section \ref{Foliations} we review the classical definitions of holomorphic foliations on $\p^2$ and the definition of the degree of a foliation in the projective space. Later we study the relatively exact meromorphic 1-forms modulo $\F(dF)$, and then conclude with the tangential center-focus problem for this foliation. \begin{theorem}
Let $\F(\omega)$ be a holomorphic foliation in $\p^2$ of degree $2d$, such that the meromorphic 1-form $\omega$ has poles along $D=\{Q=0\}$. Let $\delta_t\subset F^{-1}(t)$ be a continuous family of vanishing cycles around a center singularity, such that $\int_{\delta_t}\omega=0$. Then the form $\omega$ is written as
\begin{equation*}
\omega=F\left(\sum_{k=0}^{d}\lambda_k\frac{dR_k}{R_k}+\lambda_{d+1}\frac{dQ}{Q}\right)+dG,
\end{equation*}
where $\lambda_k\in  \C$ and $\sum_{k=0}^{d}\lambda_k+(d+1)\lambda_{d+1}=0$
and $G$ is a meromorphic function with poles along $D$.
\end{theorem}

\textbf{Acknowledgment}. I thank Hossein Movasati for introducing me to the study of monodromy action and for being a constant reference in my research. I thank Jorge Duque for his useful discussions in the first sections of the article, his points of view were very enlightening in the development of this work. I am grateful to Milo L\'opez for his support during the development of this project. The author was  supported by FAPESP grant 2022/04705-8.

\section{Homology of regular fibers}
\label{SectionHomology}
In this section we introduce the vanishing cycles for a fibration defined by a rational map on $\p^2$. Following the ideas of Lefschetz theory we show that the vanishing cycles generate the 1-homology group of a regular fiber. Furthermore, for the case where the rational map is the quotient of products of lines in general position, we calculate the dimension of the 1-homology group and provide an explicit description of the critical points associated with the vanishing cycles.\\

\noindent Let $F:\p^2\dashedrightarrow \p^1$ be a rational map defined by 
$$F(x,y,z)=\frac{P(x,y,z)}{Q(x,y,z)},$$ where $P$ and $Q$ are homogeneous polynomials of degree $d+1$. Let $N=\{P=Q=0\}$ be the set of indeterminacy. For $K\subset \p^1$  we denote by $L_K$ the preimage of $K$ via the map $F$. Note that this is equivalent to consider the blow-up of $\p^2$ along $N$, defining the preimage in the blow-up and considering the difference with the indeterminacy set \cite{hodgehossein}. 

According to Lefschetz theory, for $a$ and $b$  regular values of $F$, the homology groups satisfy $H_1(L_{\p^1\setminus a}, L_b)=0$, and $H_2(L_{\p^1\setminus a}, L_b)$ is a free module of finite dimension $\mu$ (see for example \cite[Thm 6.4]{hodgehossein}, \cite{Lamotke}
). The 2-homology $H_2(L_{\p^1\setminus a}, L_b)$ is generated by the so called \textit{Lefschetz thimbles}, and their boundaries are the \textit{vanishing cycles}. Moreover, if the critical points of $F$ are non-degenerated, then the vanishing cycles are in correspondence with the critical points of $F$, i.e. $\mu=\# \text{ of critical points of }F$. In general, if the critical points of $F$ are isolated, then we can consider a perturbation of $F$ with non degenerated critical points.
\begin{proposition}
The vanishing cycles generate the homology $H_1(L_b)$.
\end{proposition}
\begin{proof}
From the pair $(L_{P^1\setminus a}, L_b)$ we have the sequence in homology 
\begin{equation}
    \label{SequenceCoro6.1}
    \ldots\to H_2(L_{\p^1\setminus a})\to H_2(L_{\p^1\setminus a},L_b)\xto{\partial} H_1(L_b)\to H_1(L_{\p^1\setminus a})\to \ldots
\end{equation}
The relative homology $H_2(L_{\p^1\setminus a},L_b)$ is generated by the Lefschetz thimbles, and the boundary map $\partial$ transforms them into the vanishing cycles.
Thus, we study the topology of $L_{\p^1\setminus a}$. Because $a$ is a regular value, the curve $C=\{P-aQ=0\}$ is a smooth Riemann surface of genus $g=\frac{d(d-1)}{2}$. Since, $F^{-1}(a)=C\setminus N$, $$L_{\p^1\setminus a}=\p^2\setminus(N\cup F^{-1}(a))=\p^2\setminus C.$$ By Lefschetz hyperplane theorem we have $H_1(L_{\p^1\setminus a})=0$, and $H_2(L_{\p^1\setminus a})$ is a free module of finite rank,  (see for example \cite[\S 5.3]{hodgehossein}). Consequently,  the sequence in \eqref{SequenceCoro6.1} becomes 
\begin{equation}
    \label{SequenceCoro6.1a}
    \Z^{\nu} \to H_2(L_{\p^1\setminus a},L_b)\xto{\partial} H_1(L_b)\to 0.
\end{equation}
Therefore, the boundary map $\partial$ is surjective. If  $H_2(L_{\p^1\setminus a})=0$, then  $\partial$ is an isomorphism.
\end{proof}

\begin{example}
\label{exam:lineas}
Consider $F$ as the quotient of products of lines in general position, i.e.,
$$F(x,y,z)=\frac{P(x,y,z)}{Q(x,y,z)}:=\frac{\Pi_{k=0}^d R_k}{\Pi_{k=d+1}^{2d+1}R_k}\text{ , where }R_k=(2d+1-k)x+ky-k(2d+1-k)z.$$ 
The critical points of $F$ are in the affine chart $z=1$.  The critical points of $F$  are divided in three groups: The  \textit{type 1 points}, the \textit{type 2 points}, the \textit{type 3 points}. The type 1 points are associated with the intersections of the lines of $P$ or $Q$. The type 2 points are associated with the relatively compact component between the lines of $P$ or $Q$. The type 3 points are associated with the relatively compact component between two lines of $P$ and two lines of $Q$.  For example in the Figure \ref{Linesd4}, the real part of $F$ is graphed for $d=2$. The crosses in blue and green are the type 1 points, the points in blue and green are the type 2 points, the crosses in red are the type 3 points, and the points in black are the indeterminacy points.

\begin{figure}[h!]
  \centering
    \includegraphics[width=0.6\textwidth]{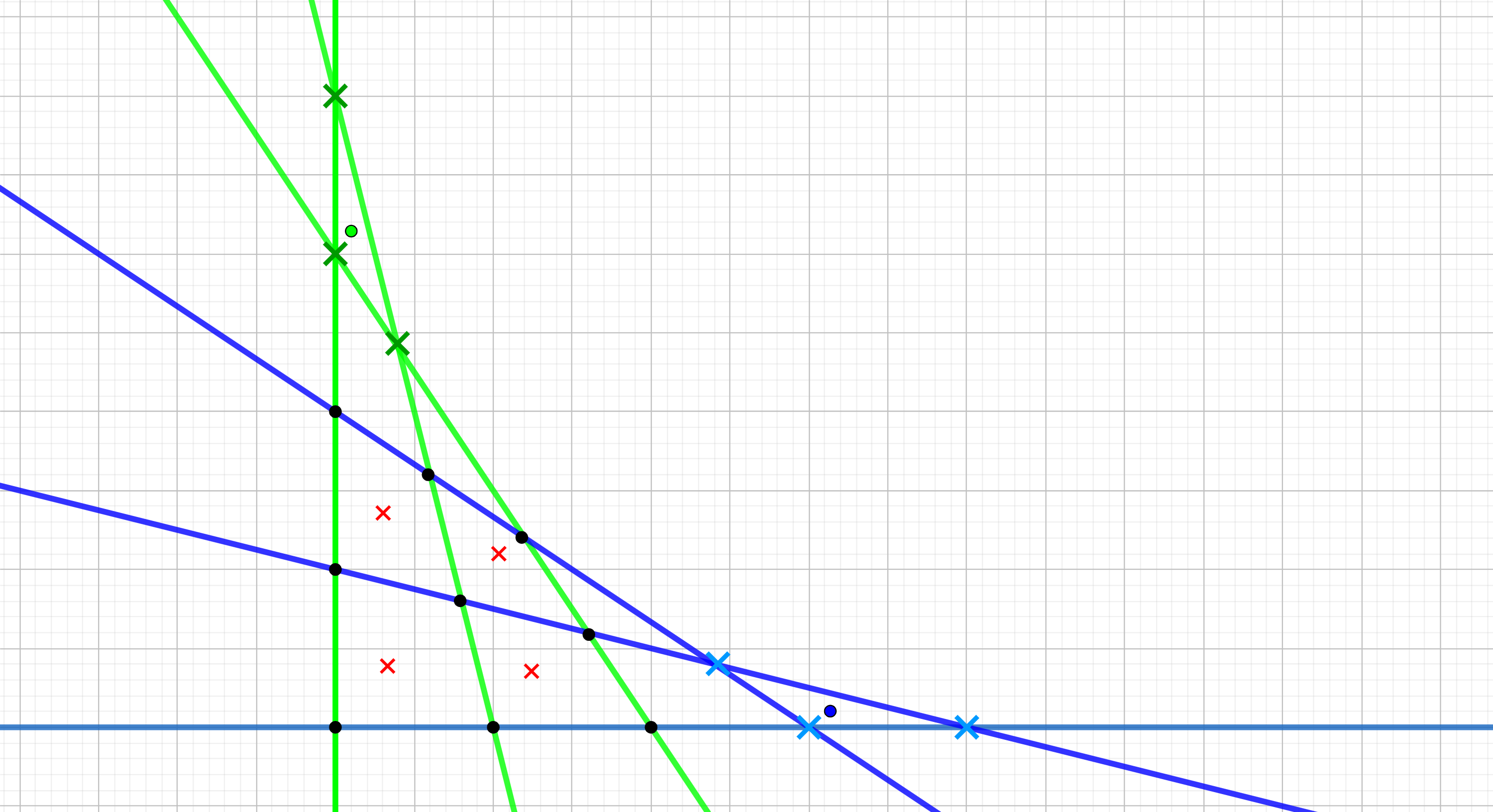}
  \caption{{Real curve  $F(x,y,1)=\frac{\Pi_{k=0}^d R_k(x,y,1)}{\Pi_{k=d+1}^{2d+1}R_k(x,y,1)}\text{ , where }R_k(x,y,1)=(2d+1-k)x+ky-k(2d+1-k),$ with $d=2$.}}
  \label{Linesd4}
\end{figure}
The number of type 1 points is $d(d+1)$, the number of type 2 points is $d(d-1)$ and the number of the type 3 points is $d^2$. Hence, the number of critical points of $F$ is $3d^2$. On the other hand, the number of indeterminacy points is $(d+1)^2$. Therefore, the dimension of the 1-homology of a regular fiber $L_b$ is $$\dim(H_1(L_b))=2g(L_b)+(d+1)^2-1=2\left(\frac{d(d-1)}{2}\right)+(d+1)^2-1=d(2d+1).$$
Note that for $d>1$ it is satisfied  $3d^2>d(2d+1)$, thus there are more vanishing cycles that the dimension of $H_1(L_b)$. Furthermore, the  kernel of the map $\partial$ in \eqref{SequenceCoro6.1a} has dimension $d(d-1)$.
\end{example}

\begin{example}
\label{d=1Homology}
For $d=1$, we have $F(x,y,1)=\frac{x(2x+y-2)}{y(x+2y-2)}.$
Thus, $C=\{P-aQ=0\}$ is a Riemann surface of genus 0. Therefore, $$L_{\p^1\setminus a}=F^{-1}(\p^1\setminus a)=\p^2\setminus(N\cup (C\setminus N))=\p^2\setminus \p^1=\C^2.$$
Then, $H_1(L_{\p^1\setminus a})=H_2(L_{\p^1\setminus a})=0$. Consequently the map $H_2(L_{\p^q\setminus a}, L_b)\xto{\partial}H_1(L_b)$, is an isomorphism. In this case, $L_b$ is a 2-sphere without the 4 points in $N$, thus the homology group $H_1(L_b)=\Z^3$ and therefore $H_2(L_{\p^1\setminus \infty}, L_b)=\Z^3.$ 
\end{example}


\section{Intersection matrix of regular fibers}
\label{seccionIntMatrix}
We compute the intersection matrix for the vanishing cycles associated with the rational map  given by the quotient of products of lines in general position. Inspired by the work of N. A'Campo \cite{ACampo} we calculate the intersection between  vanishing cycles using the real part of the graph of the rational map in the real plane. This computation is done in Proposition \ref{maintheorem1}.\\ 

\noindent Let $F:\p^2\dashedrightarrow\p$ be a rational map given by $F(x,y,z)=\frac{P(x,y,z)}{Q(x,y,z)}$, where $P,Q\in \R[x,y,z]$ are relative prime in $\C[x,y,z]$. We suppose that the critical points and indeterminacy points of $F$ are in  the real part of the local chart $z=1$, and they are isolated.  Consider the real curves given by $C_P=\{P=0\}$, $C_Q=\{Q=0\}$. Thus, the critical points of  $F$ are in one to one correspondence with the points of self intersections of $C_P$ and relatively compact components of its complement,  with the points of self intersection of $C_Q$ and relatively compact components of its complement, also, with the relatively compacts components in the intersection of $C_P$ and $C_Q$.  To define the intersection matrix to the homology $H_1(F^{-1}(b))$, we can consider a perturbation of $F$, such that the critical points are non-degenerate critical points \cite{singularnold}. 

Recall that $p$ is a critical point if $\text{grad}(F)|_p=0$. Moreover, a critical point $p$  is a  saddle  if the eigenvalues of the Hessian matrix $H_F|_p$ are one positive and one negative, and it is a real center if the eigenvalues are both positives or both negatives. Note that these definitions  are particular cases of  \textit{center points} in holomorphic foliations of $\p^2$ (or $\C^2$), which is defined in \cite{netoFwithcenter}. In fact, under local biholomorphisms  we can change a saddle to a real center; for example the fibration defined by $x^2+y^2=b$ is equivalent to the fibration defined by $xy=b$.

In Proposition \ref{lemmaACampo} we present a statement that provides an explicit computation of the intersection between two vanishing cycles that are both at saddle points or that are both at real center points. The result is inspired by \cite{ACampo}, and it is an application of the Picard-Lefschetz formula. Furthermore, in Corollary \ref{lemmaACampoCoro}, we give a criterion to show the non-intersection of two vanishing cycles.

\begin{proposition}
\label{lemmaACampo}
Let $F(x,y,z)=\frac{P(x,y,z)}{Q(x,y,z)}:\p^2\dashedrightarrow \p^1$ satisfying the above conditions.  Let  $\alpha_-$ and $\alpha_+$ be vanishing cycles defined in a regular fiber $F=b$. We suppose that $\alpha_-$ vanishes  over the critical value $-(a+\varepsilon)<0$, and $\alpha_+$  vanishes over the critical value $(a+\varepsilon)>0$, and $\varepsilon>0$. Furthermore, we assume that the map $F$ has only 0 as critical value in the interval $[-a,a]$. Moreover, we  suppose that both $\alpha_-$ and $\alpha_+$ are vanishing cycles at real center points  or both are vanishing cycles at saddle points. Then 
\begin{equation}
\label{interformulasillacentro}
\langle \alpha_-, \alpha_+\rangle=\frac{1}{2}\sum_{\alpha_l\cap \alpha_+\neq \emptyset}\langle \alpha_+,\alpha_l\rangle\langle \alpha_-,\alpha_l\rangle ,\text{ where the cycles }\alpha_l \text{ vanish over }0.
\end{equation}
\end{proposition}
\begin{proof}
 Consider the paths $\gamma_1$ and $\gamma_2$ in $\C$, from $-a$ to $a$ without passing through critical values. Moreover, we consider these paths as in Figure \ref{pathslemma1}. Thus, if we consider the complex conjugation $\overline{\gamma_1}$, then we obtain the path $\gamma_2$.

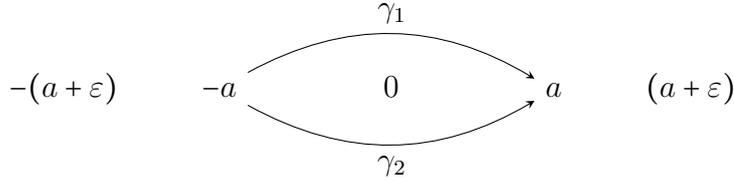
\begin{figure}[h!]
	\centering
	\begin{tikzpicture}
	\matrix (m) [matrix of math nodes, row sep=0.4em,
	column sep=2em]{
		-(a+\varepsilon)&-a & &0& &a&(a+\varepsilon)\\
	};			
	\path[-stealth]
	(m-1-2) edge [bend left=30] node [above] {$\gamma_1$}  (m-1-6) 
	edge [bend right=30] node [below] {$\gamma_2$} (m-1-6);
	\end{tikzpicture}
		\caption{\footnotesize Paths $\gamma_1$ and $\gamma_2$ from $-a$ to $a$.}
\label{pathslemma1}
\end{figure}  
\noindent Let $T_1$ and $T_2$ be the transformations $T_1,T_2: H_1(F^{-1}(-a))\to H_1(F^{-1}(a))$
induced by $\gamma_1$ and $\gamma_2$, respectively. Let $\alpha'_-$  the a vanishing cycle in the fiber $F^{-1}(a)$, which is transported to $\alpha_-$  through a path from $a$ to $b$ without monodromy. Similarly, let $\alpha'_+$ be a vanishing cycles in the fiber $F^{-1}(-a)$ transported to $\alpha_+$.  We claim that  
\begin{equation}
\label{simetriaecuacionlema}
\langle \alpha'_-, (T_1+T_2) \alpha'_+\rangle=0.
\end{equation} 
If $\alpha'_-$ and $\alpha'_+$ are vanishing cycles at center points, then the complex conjugation does not change them, that is $\overline{\alpha'_-}=\alpha'_-$ and $\overline{\alpha'_+}=\alpha'_+$. If the cycles $\alpha_-$ and $\alpha_+$ are vanishing cycles at saddle points, then $\overline{\alpha'_-}=-\alpha'_-$ and $\overline{\alpha'_+}=-\alpha'_+$. Thus, the left side of the Equation \eqref{simetriaecuacionlema} is invariant under complex conjugation. Since, this conjugation reverses the orientation of the fiber $F^{-1}(a)$, there must be a sign change in the intersection number. 

By Picard-Lefschetz formula, $T_1T_2^{-1}(T_2\alpha'_+)=T_2\alpha'_++x, $
where $x=\sum_{\alpha_l\cap \alpha_+\neq \emptyset}\langle \alpha'_+,\alpha_l\rangle\alpha_l$  and the cycles $\alpha_l$  vanish over 0. Therefore, 
\begin{align*}
    2\langle \alpha_-, \alpha_+\rangle&= 2\langle \alpha'_-, T_1 \alpha'_+\rangle \stackrel{\eqref{simetriaecuacionlema}}{=}\langle \alpha'_-, T_1\alpha'_+\rangle-\langle\alpha'_-, T_2 \alpha'_+\rangle\\
    &= \langle \alpha'_-, T_1\alpha'_+\rangle - \langle \alpha'_-, T_1T_2^{-1}(T_2 \alpha'_+)-x\rangle=\langle \alpha'_-,x\rangle.
\end{align*}

\end{proof}
\begin{corollary}
\label{lemmaACampoCoro}
Let $F(x,y,z)=\frac{P(x,y,z)}{Q(x,y,z)}:\p^2\dashedrightarrow \p^1$, as in Proposition \ref{lemmaACampo}. Let  $\alpha_-$ and $\alpha_+$ be vanishing cycles defined in the regular fiber $F=b$. We suppose that $\alpha_-$ vanishes  over the critical value $a_-$, and $\alpha_+$  vanishes over the critical value $a_+$, and $a_-<a_+$. Moreover, we  suppose that both $\alpha_-$ and $\alpha_+$ are vanishing cycles at real center points or both are vanishing cycles at saddle singularities. If there are not critical values in the interval $(a_-, a_+)$, then $$\langle \alpha_-, \alpha_+\rangle=0.$$ The same is true, if for any critical value $a_0\in(a_-, a_+)$ the vanishing cycle over $a_0$ does not intersect $\alpha_+$.
\end{corollary}
\begin{proof}
It is an immediate consequence of the formula given by the Equation \eqref{interformulasillacentro}.
\end{proof}
\noindent Another situation where two vanishing cycles do not intersect is when  they vanish at different critical points but with the same critical value. That is consequence of the definition of vanishing cycle as the unstable manifold of a vector field in a Lefschetz fibration (see for example \cite[\S 16]{seidel}\cite{DanielLopezquintic}). 
\begin{proposition}
\label{twocyclesonecriticalvalue}
Let $F:\p^2\dashedrightarrow \p^1$, with isolated singular points. Let $\alpha_1$ and $\alpha_2$ be two vanishing cycles in the  regular fiber $F^{-1}(b)$. Moreover, $\alpha_1$ vanishes in the critical point $p_1$ and $\alpha_2$ in the critical point $p_2$. If $p_1\neq p_2$ and $F(p_1)=F(p_2)=c$, then $\langle \alpha_1, \alpha_2\rangle=0.$
\end{proposition}
\begin{proof}
A vanishing cycle is defined as the points in $F^{-1}(b)$ which are going to a critical point along the flow lines of a horizontal vector field $Y$ in $F^{-1}(\mathbb{P}^1)$. This vector filed projects to a vector field in $\mathbb{P}^1$, whose integral curve along $b$ goes to  $c$. A point in the intersection $\alpha_1\cap \alpha_2$ has to go to $p_1$ and $p_2$, but it is impossible by the continuity of the flow lines of a vector field.
\end{proof}

From now on, let $F:\p^2\dashedrightarrow\p$ be the rational map given by $F(x,y,z)=\frac{P(x,y,z)}{Q(x,y,z)}=\frac{\Pi_{k=0}^d R_k}{\Pi_{k=d+1}^{2d+1}R_k}$,  where $R_k=(2d+1-k)x+ky-k(2d+1-k)z.$  As before, the critical points and indeterminacy points of $F$ are in  the real part of the local chart $z=1$, and they are isolated.  Let  $b\in\C$ be a  regular value of $F$ with $\text{Im}(b)>0$, and we consider a system of paths joining $b$ with the critical values of $F$ such that these paths lie in the upper half plane. 

We can consider a distinguished set of generators of $H_1(F^{-1}(b),\Z)$, coming from  the vanishing cycles along the chosen paths. Namely, \{$\delta^P_{i}, \Delta^P_{j}, \delta^Q_{i}, \Delta^Q_{j}, \sigma_k$\} generate $H_1(F^{-1}(b),\Z)$, where 

\begin{itemize}
    \item The vanishing cycles $\delta^P_{i}$, $i=1,\dots,\frac{d(d+1)}{2}$ vanish  in points of self intersection of $P=0.$ These critical points are saddle points.
    
    \item The vanishing cycles $\Delta^P_{j}$, $j=1,\dots,\frac{d(d-1)}{2}$ vanish in critical points in the relatively compact components between the lines of $P=0.$ These critical points are local maxima or local minima.
    
    \item The vanishing cycles $\delta^Q_{i}$, $i=1,\dots,\frac{d(d+1)}{2}$ vanish in points of self intersection of $Q=0$. These critical points are saddle points.
    
    \item The vanishing cycles $\Delta^Q_{j}$, $j=1,\dots,\frac{d(d-1)}{2}$  vanish in critical points in the relatively compact components between the lines of $Q=0.$ These critical points are local maxima or local minima.
    
    \item The vanishing cycles $\sigma_k$, $k=1,\dots, d^2$   vanish in critical points in the relatively compact components between two lines of $P=0$ and two lines of $Q=0$. These critical points are saddle points.
\end{itemize}

\noindent Note that in  Example \ref{exam:lineas}, we call \textit{type 1 points} to the critical points associated with the vanishing cycles $\delta^P_{i}$ and $\delta^Q_{i}$, \textit{type 2 points} to the critical points associated with the vanishing cycles $\Delta^P_{j}$ and $\Delta^Q_{j}$, and \textit{type 3 points} to the critical points associated with the vanishing cycles $\sigma_k$. The vanishing cycles $\{\delta^P_{i}, \Delta^P_{j}, \delta^Q_{i}, \Delta^Q_{j}, \sigma_k\}$ generate $H_1(F^{-1}(b),\Z)$. Moreover, 
the Proposition \ref{maintheorem1} provides the intersection of some of these vanishing cycles, the original result  is  proved in \cite[\S 3.2)]{ACampo}.
\begin{proposition}
\label{maintheorem1}
Let $F:\p^2\dashedrightarrow\p$ be the rational map given by $F(x,y,z)=\frac{P(x,y,z)}{Q(x,y,z)}=\frac{\Pi_{k=0}^d R_k}{\Pi_{k=d+1}^{2d+1}R_k}$,  where $R_k=(2d+1-k)x+ky-k(2d+1-k)z.$ After choosing a proper orientation for the vanishing cycles,  the intersections of the vanishing cycle are

\begin{enumerate}
    \item $\langle\Delta^P_{j},\delta^P_{i}\rangle=1,$ if the point of self-intersection associated with $\delta^P_{i}$ is a vertex of the relatively compact component associated with $\Delta^P_{j}.$ Otherwise     $\langle\Delta^P_{j},\delta^P_{i}\rangle=0.$
    
    \item $\langle\Delta^Q_{j},\delta^Q_{i}\rangle=1,$ if the point of self-intersection associated with $\delta^Q_{i}$ is a vertex of the relatively compact component associated with $\Delta^Q_{j}.$ Otherwise     $\langle\Delta^Q_{j},\delta^Q_{i}\rangle=0.$
    
    \item $\langle\Delta^P_{j_1},\Delta^P_{j_2}\rangle= 1$ if the relatively compact component associate with $\Delta^P_{j_1}$ and the relatively compact component associate with $\Delta^P_{j_2}$ have a common edge. Otherwise $\langle\Delta^P_{j_1},\Delta^P_{j_2}\rangle=0$
    
    \item $\langle\Delta^Q_{j_1},\Delta^Q_{j_2}\rangle= 1$ if the relatively compact component associate with $\Delta^Q_{j_1}$ and the relatively compact component associate with $\Delta^Q_{j_2}$ have a common edge. Otherwise $\langle\Delta^Q_{j_1},\Delta^Q_{j_2}\rangle=0$
    
    \item $\langle\delta^P_{j_1},\delta^P_{j_2}\rangle=0$ and $\langle\delta^Q_{j_1},\delta^Q_{j_2}\rangle=0$.
    
\end{enumerate}
\end{proposition}

\begin{proof}
Studying the self-intersection points in a neighborhood we can conclude the first two cases. Cases 3 and 4, where the intersection number is not zero, follows from the Proposition \ref{lemmaACampo}. When the associated relatively compact components do not have a common edge, then for a real regular value we can see that the cycles are disjoint. Case 5 follows from Proposition \ref{twocyclesonecriticalvalue}.
\end{proof}

\noindent The intersection matrix of the vectors $\{\delta^P_{i}, \Delta^P_{j}, \delta^Q_{i}, \Delta^Q_{j}, \sigma_k\}$ looks like
\begin{equation}
\label{intermatrixgeneral}
\Psi=\begin{pmatrix}
\Psi_P&*&*\\
*&\Psi_Q&*\\
*&*&*
\end{pmatrix},
\end{equation}
where $\Psi_P$ and $\Psi_Q$ are the intersection matrices associated with the polynomials $P$ and $Q$, respectively. The symbol $*$ means that the intersection numbers of these cycles are not given explicitly.
\begin{example}
\label{d=1Inter}
For $d=1$, $F(x,y,1)=\frac{x(2x+y-2)}{y(x+2y-2)}$ (Example \ref{d=1Homology}), the vanishing cycles are $\{\delta^P_1, \delta^Q_1, \sigma_1\}$, and they form a basis for $H_1(L_b)$. By Proposition \ref{maintheorem1}, the intersection matrix for the basis of the vanishing cycles is
$$\Psi=\begin{pmatrix}
0&*&*\\
*&0&*\\
*&*&0
\end{pmatrix}.$$
\end{example}
\noindent However by Corollary \ref{lemmaACampoCoro}, we conclude that  $\Psi$  is the zero matrix, which agrees with the fact that  $L_b$ is a sphere without 4 points. Namely, the homology $H_1(L_b)$ is generated by three circles around three of these points of indeterminacy. 
\begin{example}
\label{d=2Inter}
For $d=2$, $F(x,y,1)=\frac{x(4x+y-4)(3x+2y-6)}{(2x+3y-6)(x+4y-4)y}$ (Figure \ref{Linesd4}), the vanishing cycles are $$\delta_1^P, \delta_2^P, \delta_3^P, \Delta_1^P, \delta_1^Q, \delta_2^Q, \delta_3^Q, \Delta_1^Q, \sigma_1, \sigma_2, \sigma_3, \sigma_4,$$
	with critical values $0,0,0,c_1,\infty, \infty, \infty, c_2,c_3=1,c_4=1,c_5,c_6$, respectively. Note that the critical values satisfy $-\infty<c_2<c_5<c_6<c_1<0<c_3<\infty.$ By Corollary \ref{lemmaACampoCoro}, we conclude that the intersection  matrix for the vanishing cycles is
$$\Psi=\begin{pmatrix}
0&1&1&1&*&*&*&*&*&*&*&*\\
1&0&0&0&*&0&0&0&0&0&*&*\\
1&0&0&0&*&0&0&0&0&0&*&*\\
1&0&0&0&*&0&0&0&0&0&*&*\\
*&*&*&*&0&1&1&1&*&*&*&*\\
*&0&0&0&1&0&0&0&0&0&*&*\\
*&0&0&0&1&0&0&0&0&0&*&*\\
*&0&0&0&1&0&0&0&0&0&*&*\\
*&*&*&*&*&*&*&*&0&0&*&*\\
0&0&*&*&0&0&*&*&0&0&*&*\\
0&0&*&*&0&0&*&*&*&*&0&0\\
0&0&*&*&0&0&*&*&*&*&0&0
\end{pmatrix}.$$
\end{example}
As shown in the Examples \ref{d=1Inter} and \ref{d=2Inter}, under some assumptions about the critical values it is possible to use the Corollary \ref{lemmaACampoCoro} to determine several values in the intersection matrix that are zero. Although in general we do not completely determine the intersection matrix, as seen below, with the description in \eqref{intermatrixgeneral} it is enough.

\section{Monodromy problem}
\label{seccionOrbits}
In this section we compute the submodule of the 1-homology group of a regular fiber of $F$ generated by the monodromy action on a vanishing cycle. We show that the 1-homology group of the compact fiber  is generated by some specific vanishing cycles. Then, in Theorem \ref{maintheorem2} we prove that the orbit of the monodromy action on any vanishing cycle generates the 1-homology of the compact fiber. This result is the equivalent of the monodromy problem for a rational map given by the quotient of the products of lines in general position.\\

\noindent As we know, for a regular value $b$ the fiber $L_b$ is a Riemann surface of genus  $g(L_b)=\frac{d(d-1)}{2}$ without the $(d+1)^2$ points of indeterminacy $N=\{P=Q=0\}$. We can define the compact fiber $\overline{L_b}:=L_b\cup N.$ Thus, $\overline {L_b}$ is a Riemann surface  of genus $g(L_b)$. Consequently, we have 
$$
    \dim (H_1(L_b))=d(2d+1), \qquad
    \dim (H_1(\overline{L_b}))=d(d-1).
$$
For $H_1(L_b)$ we can consider a basis given by a basis of $H_1(\overline{L_b})$ together with small circles around of the points of $N$. Let $I$ be the subspace of dimension $(d+1)^2-1$ generated by the cycles around of points of $N$, then $H_1(L_b)=H_1(\overline{L_b})\oplus I.$ From the definition of $I$ it follows that
\begin{equation}
\label{CaracterizacionI}
    I=\{\delta\in H_1(L_b)\text{ }|\text{ }\langle \delta, \delta'\rangle=0, \text{ }\forall \delta'\in H_1(L_b)\}.
\end{equation}
\begin{proposition}
\label{baseparalacompacta}
The vanishing cycles $\{\delta^P_i, \Delta_j^P\}$, $i=1,\ldots , \frac{d(d+1)}{2}$, $j=1,\ldots, \frac{d(d-1)}{2}$ generate $H_1(\overline{L_b}).$
\end{proposition}
\begin{proof}
Let denote by $\{\alpha_i\}$, $i=1,\ldots, d(d-1)$ a basis of $H_1(\overline{L_b})$ and $\{\beta_j\}$,    $j=1,\ldots, (d+1)^2-1$ a basis of $I$. Thus, any vanishing cycles can be written as a linear combination of $\alpha_i$ and $\beta_j$. From \eqref{CaracterizacionI} we have that the intersection between two vanishing cycles depends on only of the coefficients that appear accompanying the $\alpha_i$'s. Therefore, $$\langle \phi(\delta^P_i), \phi(\Delta^P_j)\rangle=\langle \delta^P_i, \Delta^P_j\rangle,
\qquad \langle \phi(\Delta^P_i), \phi(\Delta^P_j)\rangle=\langle \Delta^P_i, \Delta^P_j\rangle,$$
where  $\phi$ is the map $H_1(L_b)\xto{\phi} H_1(\overline{L_b})$  which takes the cycles in $I$ to zero in $H_1(\overline{L_b})$. Consequently, the submatrix $\Psi_P$ of the intersection matrix $\Psi$ is an intersection matrix of $H_1(\overline{L_b})$. Moreover, the rank of $\Psi_P$ is $d(d-1)$. Hence, there are $d(d-1)$ vanishing cycles such that their images via the map $\phi$ are linearly independent in $H_1(\overline{L_b})$. 
\end{proof}

\begin{remark}
The Proposition \ref{baseparalacompacta} is also true for the vanishing cycles $\{\delta^Q_i, \Delta_j^Q\}$, $i=1,\ldots , \frac{d(d+1)}{2}$, $j=1,\ldots, \frac{d(d-1)}{2}$. On the other hand, we can consider an order in $\delta_i^P$, $i=1, \ldots, \frac{d(d+1)}{2}$ as follows: In $z=1$, on the first line $R_0$, we establish $\delta_1^P, \delta_2^P,\ldots, \delta_d^P$ being the intersection of $R_0$ with $R_d, R_{d-1}, \ldots, R_1$, respectively. Then, we continue with the line $R_1$, thus the cycles $\delta_{d+1}^P,\ldots, \delta_{2d-1}^P$ are the intersections of $R_1$ with $R_d, \ldots, R_{2}$, and so on.  A basis of $H_1(\overline{L_b})$ can be given by considering the vanishing cycles $\{\delta^P_i, \Delta_j^P\}$, $i=1,\ldots , \frac{d(d+1)}{2}$, $j=1,\ldots, \frac{d(d-1)}{2}$ without the cycles $\delta_d^P, \delta_{2d-1}^P,\ldots, \delta_{kd-\frac{k(k-1)}{2}}^P\ldots, \delta^P_{\frac{d(d+1)}{2}}$. In fact, the number of elements in this set of vanishing cycles is $\frac{d(d+1)}{2}+\frac{d(d-1)}{2}-d=d(d-1)=\dim H_1(\overline{L_b})$. Moreover,  using Gauss-Jordan elimination in the intersection matrix $\Psi_P$, it is possible  to show that these vectors are linearly independent.
\end{remark}

As we mentioned before, the monodromy problem for genetic lines in $\C^2$ is studied in \cite{hosseincenterlog}, where
H. Movasati presents the next result,
\begin{theorem}[H. Movasati]
	\label{monodromyHossein}
	Let $f$ be a polynomial given by $$f(x,y):=P(x,y,1)=\Pi_{k=0}^d (2d+1-k)x+ky-k(2d+1-k),$$ for any regular value $b$, the action of the monodromy on any  vanishing cycle generates the homology $H_1(\overline{f^{-1}(b)})$. Here, $\overline{f^{-1}(b)}=f^{-1}(b)\cup \{(d+1) \text{ points of indeterminacy}\}$.
\end{theorem}

\noindent For a generic $\frac{P(x,y,z)}{Q(x,y,z)}$ rational map, the monodromy action is transitive [Cor. 3.1.2.]\cite{hosseininonthetopology}. The map $F(x,y,z)=\frac{\Pi_{k=0}^d R_k}{\Pi_{k=d+1}^{2d+1}R_k}\text{ , where }R_k=(2d+1-k)x+ky-k(2d+1-k)z$ is not generic, however, the intersection matrix does not change for small perturbations of $F$. Thus, given a vanishing cycle $\delta$ and $\Delta^P_{j_0}$ with critical value $c_{j_0}$,  some element in the subspace generated by the orbit of $\delta$  intersects some  vanishing cycles associated with the vertices of the relatively compact component associated with $\Delta^P_{j_0}$. Since  $c_{j_0}$ is different to the other critical values, then  the subspace generated by the orbit of the monodromy action on $\delta$ contains $\Delta_{j_0}^P$. Therefore, any vanishing cycle contains all the vanishing cycles $\Delta^P_j$, $j=1,\ldots, \frac{d(d-1)}{2}$.


The goal in this section is to conclude that the subspace generated by the orbit of the monodromy action on any vanishing cycle  is $H_1(\overline{L_b})$.  According to Proposition \ref{baseparalacompacta} and the discussion of the previous paragraph, in order to conclude that, it is enough to restrict our attention to the vanishing cycles $\{\delta_i^P, \Delta_j^P\}$, $i=1,\ldots, \frac{d(d+1)}{2}$, $j=1, \ldots, \frac{d(d-1)}{2}$. With this observation, we are in the case of  
 Theorem \ref{monodromyHossein}.
\begin{theorem}
\label{maintheorem2}
Let  
$F:\p^2\dashedrightarrow\p$  be a rational map given by $$F(x,y,z)=\frac{P(x,y,z)}{Q(x,y,z)}=\frac{\Pi_{k=0}^d (2d+1-k)x+ky-k(2d+1-k)z}{\Pi_{k=d+1}^{2d+1}(2d+1-k)x+ky-k(2d+1-k)z},$$ for any regular value $b$ of $F$, the action of the monodromy on any vanishing cycle  generates the homology $H_1(\overline{L_b})$.
\end{theorem}
\begin{proof}
From the intersection matrix given by \eqref{intermatrixgeneral}, and the Picard-Lefschetz formula, we conclude that the monodromy action on $H_1(f^{-1}(b))$ in Theorem \ref{monodromyHossein}, is the same as the monodromy action on the subspace of $H_1(L_b)$ generated by $\{\delta^P_i, \Delta_j^P\}$. By Proposition \ref{baseparalacompacta}, we know that this subspace is $H_1(\overline{L_b})$. Finally we conclude the proof,  observing that $\dim(H_1(\overline{f^{-1}(b)}))=\dim(H_1(\overline{L_b}))$ and using Theorem \ref{monodromyHossein}.
\end{proof}

\section{Brieskorn lattices/Petrov Modules}
\label{Brieskorn}
So far we have developed the topological aspects associated with the fibration  $F:\p^2\dashedrightarrow\p$  given by  given  by $F(x,y,z)=\frac{P(x,y,z)}{Q(x,y,z)}=\frac{\Pi_{k=0}^d R_k}{\Pi_{k=d+1}^{2d+1}R_k}$,  where $R_k=(2d+1-k)x+ky-k(2d+1-k)z.$ In this section we study the algebraic aspect coming from the Brieskorn modules and Gauss-Manin connection associated with $F$. The main aim of this section is describe the kernel of the Gauss-Manin connection for $F$.\\

\noindent Since we are working in the projective space $\p^2$, it is necessary to introduce the Brieskorn modules with poles as in  \cite{yadollahBrieskorn}. We denote by $D$ the fiber over infinity, that is $D=F^{-1}(\infty)=\{Q=0\}$. Let $\Omega^i(*D)$ be the set of rational $i$-forms with pole along  $D$. Let $t\in \C=\p^1\setminus\{\infty\}$ be the affine coordinate. The set $\Omega^i(*D)$ is a $\C[t]-$module in the following sense  $$p(t)\cdot \omega=p(F)\omega, \hspace{4mm}\omega\in \Omega^i(*D)\text{, }p\in \C[t].$$ The Brieskorn lattice/Petrov module is defined as $$H=\frac{\Omega^1(*D)}{d\Omega^0(*D)+dF\wedge\Omega^0(*D)},$$ it is a $C[t]-$module. Moreover, we define $$V=\frac{\Omega^2(*D)}{dF\wedge \Omega^1(*D)}.$$ According to \cite{yadollahBrieskorn, Hosseincoh}, $V$ is a $\C-$vector space of finite dimension. Let $A$ be the linear map associated with the multiplication by $F$ in $V$, and let $p_A$ be the minimal polynomial of $A$. Let $\tilde C=\{c_0=0, c_1=1,c_2,\ldots, c_{2d(d-1)}\}$ be the set of non infinite critical values of $F$, and let $\tilde H$ be the localization of $H$ by polynomials in $t$ with zeros on $\tilde C$, i.e., $$\tilde H=\bigg\{\frac{\omega}{s(t)}\text{ }\bigg|\text{ }\omega\in H\text{, }s\in \C[t]\text{, }Z(s)\subset \tilde C\bigg\}.$$ 
The \textit{Gauss-Manin} connection $\nabla:H\to \tilde H$ is defined  equivalently as in \cite{hosseincenterlog}. If $\omega \in H$, then $d\omega \in V$ and  $p_A(F)d\omega=0$ in $V$. Therefore, there exists $\eta\in \Omega^1(*D)$ such that 
$$p_A(F)d\omega=dF\wedge \eta.$$
Thus, we define $\nabla \omega:=\frac{\eta}{p_A(F)}.$ The well definition of $\nabla$ follows directly from \cite{hosseincenterlog}. Note that the polynomial $p_A=p_A(t)$, where $t$ as before.  Moreover, it is possible to extent $\nabla$ to $\tilde H$ by using the rule $$\nabla\left(\frac{\omega}{s(t)}\right)=\frac{s(t)\nabla \omega-\frac{\partial s}{\partial t}\omega}{s^2}.$$
As in \cite{hosseincenterlog}, there is an isomorphism of $\tilde H$, with the differential forms with poles along the irreducible components of all critical fibers of $F$ other than the fiber at the infinity. Moreover, let $b$ be a regular value of $F$ and let  $\delta_t\subset F^{-1}(t)$ be a continuous family of cycles with $t\in (\C, b)$. For $\omega \in \tilde H$, if $\int_{\delta_t}\omega$ is well-defined, then (see \cite{singularnold},\cite{hosseincenterlog})
\begin{equation}
\label{relaciontopoalg}
\frac{d}{dt}\int_{\delta_t}\omega=\int_{\delta_t}\nabla \omega.
\end{equation}
Coming up next we compute the kernel of $\nabla^2=\nabla \circ \nabla$. This computation is a consequence of Proposition \ref{closeform}, which is a classic result in holomorphic foliation, see for example
\cite[Prop. 2.5.11]{NetoScarduabook}.

\begin{proposition}
	\label{closeform}
	Let $\omega\neq 0$ be a meromorphic closed $1-$form in $\p^n$. Then 
	$$\omega=\sum_{j=1}^k \lambda_j \frac{df_j}{f_j}+d\left(\frac{g}{f_1^{r_1-1}\cdots f_{k}^{r_k-1}}\right),$$
	where
	\begin{enumerate}
		\item $k\geq 2$ and $f_1,\ldots, f_k, g$ are homogeneous polynomials in $\C^{n+1}.$
		\item $f_1,\ldots, f_k$ are irreducible.
		\item $\lambda_1, \ldots, \lambda_k\in \C$ and $\sum_{j=1}^k\lambda_j \deg(f_j)=0.$
		\item If $r_j>1$, then $f_j$ does not divide $g$.
		\item If $r_j=1$, then $\lambda_j\neq 0$.
		\item $\deg(g)=\deg(f_1^{r_1-1}\cdots f_k^{r_k-1}).$
		\item The set of poles of $\omega$ is  $\bigcup_{j=1}^k\{f_j=0\}$. Furthermore, $r_j$ is the order of $\{f_j=0\}$ as pole of $\omega$.
	\end{enumerate}
\end{proposition}
\noindent  The fiber over the critical value $c_0=0$ is reducible, namely $R_0R_1\cdots R_d=0$. Also, the fiber over the critical value $c_1=1$ is reducible; because $P-Q=(x-y)S(x,y,z)$, with $\deg(S)=d$. The other fibers of $F$ over $\C=\p^1\setminus \infty$ are irreducible. We can consider a perturbation of $F$ such that the fiber over $c_1=1$ is irreducible. For example, if we consider $\tilde R_{2d+1}=(2d+1)y+\varepsilon$ for a $\varepsilon>0$ small enough, then $\tilde F(x,y,z)=\frac{R_0R_1\cdots R_d}{R_{d+1}R_{d+2}\cdots \tilde R_{2d+1}}$ has reducible fibers only over $c_0=0$ and over $\infty$. From here on, to avoid burdening the notation, we denote by $F$ said perturbation $\tilde F$. Note that intersection matrix associated with this new $F$ does not change, therefore the results in \S \ref{seccionIntMatrix} and \S \ref{seccionOrbits} still hold for this $F$.

In $\tilde H$ we consider the subset generated by the  logarithmic forms with pole along $D$ 
$$\mathcal{L}=\bigg\{\omega \in \tilde H\text{ }\bigg|\text{ }\omega=\sum_{k=0}^{d}\lambda_j\frac{dR_k}{R_k}+\lambda_{d+1}\frac{dQ}{Q},\hspace{1mm}  \lambda_k\in \C,\hspace{1mm}\sum_{k=0}^{d}\lambda_k+(d+1)\lambda_{d+1}=0\bigg\}.$$
The condition $\sum_{k=0}^{d}\lambda_k+(d+1)\lambda_{d+1}=0,$ ensures that $\omega$ defines a foliation in $\p^2$. Namely, for $\omega$ to define a foliation in $\p^2$ it is necessary that $\imath_{E} \omega=0$ where $E=x\frac{\partial}{\partial x}+y\frac{\partial}{\partial y}+z\frac{\partial}{\partial z}$ is the radial vector field  \cite{AlcidesHolFol2AlgCurve}. Since $\imath_E(dR_j)=R_j$ and $\imath_E(dQ)=(d+1)Q$, $\imath_E(\omega)=\sum_{j=0}^{d}\lambda_j+(d+1)\lambda_{d+1}$. On the other hand, note that $\sum_{j=0}^d \lambda_j\frac{dR_j}{R_j}$ corresponds to the irreducible components of the fiber $F=0$. Since for any  $0\neq c_i\in \tilde C$ the polynomial $P-c_iQ$ is irreducible, the logarithmic forms associated with other fiber of $F$ are 0 in $\tilde H$.
\begin{proposition} 
	\label{kernelnabla2}
If $\omega \in \ker(\nabla^2)$, then $\omega=F\eta_1+\eta_0$, where $\eta_0, \eta_1\in \mathcal{L}$. Moreover, if  $\omega \in \ker(\nabla^2\cap H)$, then
\begin{equation}
\label{forminkernel}
	\omega=F\left(\sum_{k=0}^{d}\lambda_k\frac{dR_k}{R_k}+\lambda_{d+1}\frac{dQ}{Q}\right)
\end{equation}
where $\lambda_k\in  \C$ and $\sum_{k=0}^{d}\lambda_k+(d+1)\lambda_{d+1}=0$.
\end{proposition}
\begin{proof}
Suppose that $\nabla \omega$ is zero in $\tilde H$, that is $\nabla \omega=dA+BdF$, where $A,B$ are rational functions on $\p^2$ with poles along the irreducible components of all critical fibers of $F$. By definition of Gauss-Manin connection, $d\omega=dF\wedge \nabla \omega$, then $d\omega=dF\wedge dA$. Thus, $d(\omega+AdF)=0$ and by using Proposition \ref{closeform}, we conclude $\omega-AdF=\eta+d\left(\frac{g}{h}\right),$
with $\eta\in \mathcal{L}$. 
Therefore, $w=\eta$ in $\tilde H$.

For $\omega\in \ker(\nabla^2)$, we have $\nabla(\nabla(\omega))=0$, then by the previous result, $\nabla \omega=\eta_1$ where $\eta_1\in \mathcal{L}$. Moreover, since $\eta_1=\nabla(F\eta_1)$, then $\nabla(\omega-F\eta_1)=0$. Consequently, $\omega=F\eta_1+\eta_0$, where $\eta_0, \eta_1\in\mathcal{L}$. 
If in addition $\omega\in H$, then 
$\eta_0=0$.
\end{proof}



\section{Tangential center-focus problem}
\label{Foliations}
We conclude this article with an application of Theorem \ref{maintheorem2} to holomorphic foliations on $\p^2$. We initially review what a holomorphic foliation on $\p^2$ is, and how the degree of a foliation is defined. Then  we give a characterization of the relatively exact meromorphic 1-forms modulo $dF$, where $F$ is the rational map given by the quotient of products of lines in general position. Finally, the tangential center-focus problem is solved for this rational map in Theorem \ref{maintheorem3}.\\

\noindent A holomorphic foliation in $\p^2$ can de defined as follows (see for example \cite{AlcidesHolFol2AlgCurve}), in affine coordinates $(x,y)\in \C^2$, the foliation $\F=\F(\tilde \omega)$ is defined by a polynomial 1-form $\tilde \omega=R(x,y)dy-S(x,y)dx$. The singular points of $\F$ are the set $\{R=S=0\}$, and the leaves of the foliation are the solutions of $\tilde \omega=0$. Let $\pi: \p^2\setminus \{z=0\}\to \C^2$ be the map given by $(x,y,z)\to \left(\frac{x}{z}, \frac{y}{z}\right)$. Then the pullback $\pi^*\tilde \omega$ 
 has poles at $z=0$. Hence, it is possible to write $\pi^*\tilde \omega=z^{-k}\omega_0$, where $\omega_0$ is holomorphic, and $k$ is chosen in such a way that such $\omega_0$ is not divisible by $z$. 
 The classical definition of the degree of the foliation $\F(\tilde \omega)$ is $\deg(\F):=k-2$.   
 
 A canonical way to define a foliation in $\p^2$, without using a coordinate system, is the following. Let $\{U_i\}_{i\in I}$ be an open covering of $\p^2$, and let $\alpha_i$ be a collection of holomorphic 1-forms on $U_i$. Moreover, let $c_{ij}$ be holomorphic functions without zero defined on $U_i\cap U_j$, such that $\alpha_i=c_{ij}\alpha_j$, and $c_{ij}$ satisfies cocycle conditions. Thus the transition maps  $\{c_{ij}\}_{i,j\in I}$ define a line bundle $L$. Using this information we can define a foliation on $\p^2$. Thus, a foliation $\F=\F(\alpha)$ on $\p^2$  is given by $\alpha\in H^0(\p^2, \Omega^1\otimes L)$, where $\Omega^1$ is the cotangent bundle of $\p^2$. 
 
 As in \cite{hosseinabelian}, we want to define a foliation  $\F(\alpha)$ where $\alpha\in H^0(\p^2, \Omega^0\otimes L)$, as a meromorphic 1-form. For this we choose a non-zero
 section $s\in H^0(\p^2, L)$ and then define the meromorphic 1-form $\omega=\frac{\alpha}{s}$ in $\p^2$. We denote the foliation indistinctly as $\F(\alpha)$ or  $\F(\omega)$. A special family of foliations in which we are interested are the \textit{Hamiltonian foliations} or \textit{integrable foliations} (foliations with first integral). For a foliation $\F=\F(\alpha)$ with $\alpha \in H^0(\p^2, \Omega^2\otimes L)$, and $s\in H^0(\p^2, L)$, if $\frac{\alpha}{s}=df$  for a meromorphic function $f$ on $\p^2$, then $f$ is called \textit{first integral} of $\F$. If $\frac{\alpha}{s}$ is closed, then  the meromorphic section $s$ is called \textit{integrating factor}.
 
 Associated with the meromorphic section $s$ it is possible to define the divisor $$\text{div}(s)=(s)_0-(s)_{\infty},$$
 thus the \textit{degree} of the foliation is defined as $\deg(\F(\omega))=\deg(\text{div}(s))-2$. 

 \begin{example}
 	Let $\alpha=pBdA-qAdB$, where $A,B$ are homogeneous polynomials in $\C^3$ where $\deg(A)=a+1$, $\deg(B)=b+1$ and $\frac{q}{p}=\frac{a+1}{b+1}$. A first integral of $\F(\omega)$ is $f=\frac{A^{p}}{B^{q}}$, with integrating factor $s=\frac{B^{q+1}}{A^{p-1}}$. Therefore, $\deg(\F)=(b+1)(q+1)-(a+1)(p-1)-2=a+b$.
 \end{example}
 
 \begin{example} Let $\alpha=f_1\cdots f_n\sum_{k=1}^n \lambda_k\frac{df_k}{f_k}=\sum_{k=1}^nf_1\cdots \hat{f_{k}}\cdots f_n \lambda_k df_k$, where $f_j$ are homogeneous polynomials in $\C^3$ with $\deg(f_j)=d_j$, and  $\lambda_j\in \C^{*}$ satisfy $\sum_{j=1}^n \lambda_j d_j=0$. The foliation $\F(\omega)$ has not first integral, in fact this family of foliations is called \textit{logarithmic foliations}. Note that $s=f_1\cdots f_n$ is a integrating factor because $d(\frac{\alpha}{s})=d\left(\sum_{j=1}^n \lambda_j\frac{df_j}{d_f} \right)=0$. Therefore, $\deg(\F)=\sum_{j=1}^n d_j-2.$ 
 \end{example}
 
 \begin{remark}
 	\label{degreeorder}
 	Note that if the section $s$ has no poles, then the order of the pole of the meromorphic form $\omega=\frac{\alpha}{s}$ defining the foliation $\F=\F(\omega)$ is $\deg(\F)+2$, and the pole divisor of $\omega$ is $(s)_0$. This happens when $\alpha$ has no zero-divisor, and we can suppose this by changing the line bundle $L$ (see \cite{hosseinabelian}). 
 \end{remark}
\noindent Let $\omega_1$ be a  meromorphic  1-form. We say that $\omega_1$ is relatively exact modulo the foliation $\F$ if the restriction of $\omega_1$ to any leaf of $\F$ is exact. That is, for any leaf $L$ of $\F$, there exists a meromorphic function $g$ on $L$ such that $\omega_1|_L=dg$. Again, 	let $F(x,y,z)=\frac{P(x,y,z)}{Q(x,y,z)}=\frac{\Pi_{k=0}^d R_k}{\Pi_{k=d+1}^{2d+1}R_k}$,  where $R_k=(2d+1-k)x+ky-k(2d+1-k)z$ for $k=0,\ldots,2d$, and $R_{2d+1}=(2d+1)y+\varepsilon$ for a $\varepsilon>0$ small enough.   From here on, we consider the foliation $\F$ as the foliation defined by the 1-form $dF$. 
For any integrable foliation, in particular for $\F(dF)$, it is easy to see that a meromorphic 1-form $\omega_1$ is relatively exact modulo $\F$ if and only if $\int_{\delta}\omega=0,$
for any loop $\delta$ in the leaves of $\F$, where the integral is well-defined.

 There are several works where there are  characterizations of relatively exact 1-forms  modulo  integrable foliations, for example the reader can see \cite{gavrilovpetrov, mucino, hosseinabelian, yadollahBrieskorn}. Inspired by these works, in Proposition \ref{formasrelativamenteexactas} we show for our particular case  $\F=\F(dF)$ a characterization of the relatively exact meromorphic 1-forms modulo $\F$. 

\begin{proposition}
	\label{formasrelativamenteexactas}
	Let $\omega$ be a meromorphic 1-form with pole divisor $nD$, where $D=\{Q=0\}$. If $\omega$ is relatively exact modulo $\F(dF)$, then $$\omega=dG+TdF,$$
	where $G$ and $T$ are meromorphic function with pole divisor $nD$ and $(n-2)D$, respectively.
\end{proposition}
\begin{proof}
	Let $L$ be the line in $\p^2$ defined by $\{z=0\}$ and let $U=\p^2\setminus N$, where $N$ are the points of indeterminacy of $F$. For $u\in U$, we consider  the set of points $F^{-1}(F(u))\cap L=\{p_1, p_2, \ldots, p_{d+1}\}$. Then, it is possible to set the map $G:\p^2 \setminus N\to \C$, defined by $$G(u)=\frac{1}{d+1}\left(\sum_{i=1}^{d+1}\int_u^{p_i}\omega\right),$$
where the integration is performed over a path in the fiber that passes through $u$, starting at $u$ and ending at $p_i$. The map $G$ does not depend on the path choice because $\int_{\delta}\omega=0$ for any loop $\delta \subset F^{-1}(F(u)) $. Moreover $G$ is holomorphic in $\p^2\setminus D$. In order to prove that $G$ is meromorphic in $\p^2$, by Levi's extension theorem (see for example \cite{FritzscheGrauert}[Ch4, \S 4], \cite{NetoScarduabook}[\S 7.3]), it is enough to show that $G$ is meromorphic in $U$. For $u\in D$, we have $$\int_u^{p_i}\omega=F^{n}\int_u^{p_i}F^{-n}\omega,$$
where $F^{-n}\omega_1$ is holomorphic. Then, this integral has order pole at most $n$ along $D$.

On the oder hand,  note that 
\begin{align*}
dG\wedge dF=\frac{1}{d+1}\left(\sum_{i=1}^{d+1}d\left(\int_u^{p_i}\omega \right)\wedge dF \right)=\omega\wedge dF,
\end{align*}
then $(dG-\omega)\wedge dF=0$. Therefore, there exits a rational function $T$, such that $TdF=\omega-dG$. Finally, we conclude by observing that $dF$ has pole divisor $2D$. 
\end{proof}
\noindent It is necessary to show that the integrals $\int_{\delta_t}\nabla^2\omega$ over cycles around indeterminacy points are zero, this will be shown in Proposition \ref{IntegralzeroinIdeterminancy}.
The following Lemma will be used in the proof of this proposition.

\begin{lemma}
	\label{formasobreFholo}
Let $\F(\omega)$ be a holomorphic foliation in $\p^2$ of degree $2d$, such that the meromorphic 1-form $\omega$ has poles along $D=\{Q=0\}$. Then the 1-form $\omega\cdot\left(\frac{ Q}{F}\right)$ is holomorphic in $\{Q=0\}\setminus N$.
\end{lemma}
\begin{proof}
	By Remark \ref{degreeorder}, it is possible to write $ \omega=\frac{\alpha}{Q^2}$, with $\alpha$ as before. Thus,
	 $\omega\cdot\left(\frac{ Q}{F}\right)=\frac{\alpha} {P}.$
 Moreover, $\{P=0\}\cap \{Q=0\}\subset N$.
\end{proof}
\noindent Inspired by \cite{hosseincenterlog}, we show that the integrals of $\nabla^2\omega$ along loops encircling the  points of indeterminacy $N$ are zero. 
\begin{proposition}
	\label{IntegralzeroinIdeterminancy}
Let $\F(\omega)$ be a holomorphic foliation in $\p^2$ of degree $2d$, such that the meromorphic 1-form $\omega$ has poles along $D=\{Q=0\}$. Let $c$ be an indeterminacy point, and let  $\{\delta_t\subset F^{-1}(t)\}_{t\in\C}$ be a continuous family of cycles around $c$. Then $\int_{\delta_t}\nabla^2\omega=0.$
\end{proposition}
\begin{proof}
	We define the holomorphic function  $h(t)=\int_{\delta_t}\omega$ in $\C$. At $t=\infty$, we have $F^{-1}(t)=\left(\frac{F}{Q}\right)^{-1}(t)$, consequently 
	$$\frac{h(t)}{t}=\int _{\delta_t}\frac{\omega}{F}=\int _{\delta_t}\omega\left(\frac{Q}{F}\right),$$ 
	and by Lemma \ref{formasobreFholo}  this integral is finite. Therefore, $h(t)$ is a polynomial of degree at most  $1$. Finally, by the definition of Gauss-Manin connection, we have
	$$\int_{\delta_t}\nabla^2\omega=\frac{d^2}{dt^2}h(t)=0$$
	\end{proof}
 \noindent At this point we already have the tools to show the main theorem of this section, namely in Theorem \ref{maintheorem3} we solve the tangential center-focus problem for the rational map $F$ defined as the quotient of the product  of lines in general position, $F(x,y,z)=\frac{P(x,y,z)}{Q(x,y,z)}=\frac{\Pi_{k=0}^d R_k}{\Pi_{k=d+1}^{2d+1}R_k}$,  where $R_k=(2d+1-k)x+ky-k(2d+1-k)z$ for $k=0,\ldots,2d$, and $R_{2d+1}=(2d+1)y+\varepsilon$ for a $\varepsilon>0$ small enough. 
\begin{theorem}
\label{maintheorem3}
Let $\F(\omega)$ be a holomorphic foliation in $\p^2$ of degree $2d$, such that the meromorphic 1-form $\omega$ has poles along $D=\{Q=0\}$. Let $\delta_t\subset F^{-1}(t)$ be a continuous family of vanishing cycles around a center singularity, such that $\int_{\delta_t}\omega=0$. Then the form $\omega$ is written as
\begin{equation*}
	\omega=F\left(\sum_{k=0}^{d}\lambda_k\frac{dR_k}{R_k}+\lambda_{d+1}\frac{dQ}{Q}\right)+dG,
\end{equation*}
where $\lambda_k\in  \C$ and $\sum_{k=0}^{d}\lambda_k+(d+1)\lambda_{d+1}=0$
and $G$ is a meromorphic function with poles along $D$.
\end{theorem}
\begin{proof}
For a regular value $b$ and $t\in(\C, b)$, from \eqref{relaciontopoalg}, we have $\int_{\delta_t}\nabla^2 \omega=\frac{d^2}{dt^2}\int_{\delta_t}\omega=0,$ and by using Theorem \ref{maintheorem2} we conclude that $\int_{\delta}\nabla^2\omega=0$ for any $\delta\in H_1(\overline{F^{-1}(t)})$. For a cycle $\delta'$  in $I$, by Proposition \ref{IntegralzeroinIdeterminancy}, the integral over $\delta'$ is also zero. Therefore, since $H_1(F^{-1}(t))=H_1(\overline{F^{-1}(t)})\oplus I$, $\nabla^2 \omega$ is relatively exact modulo the foliation given by $dF$.

By Proposition \ref{formasrelativamenteexactas} and because  $\omega$ has pole along $D$ we conclude  that $\omega$ is in  $\ker(\nabla^2\cap H)$. Hence, from Proposition \ref{kernelnabla2}, we have 
$$	\omega=F\left(\sum_{k=0}^{d}\lambda_k\frac{dR_k}{R_k}+\lambda_{d+1}\frac{dQ}{Q}\right)+d\tilde G+TdF,$$
where $\lambda_k \in \C$ and  $\sum_{k=0}^{d}\lambda_k+(d+1)\lambda_{d+1}=0$. Moreover, $\tilde  G$ and $T$ are meromorphic functions with pole along $D$. Since $\deg(\omega)=2d$ and $\deg(Q)=d+1$, then the order of the pole of $\tilde G$ along $D$ is at most  2, and the order of the pole of $T$ along $D$ is 0. Furthermore, $TdF=\frac{T}{Q^2}(QdP-PdQ),$
then $2(d+1)-\deg(T)-2=2d$. Thus, $\deg(T)=0$, and so  $G$ is defined as $\tilde  G+TF$.
\end{proof}



\clearpage
\begin{normalsize}
	\bibliographystyle{abbrv}
   \bibliography{References}
\end{normalsize}

\bigskip

\bigskip

\sf{\noindent Daniel L\'opez Garcia\\
	Instituto de Matem\'atica e Estat\'istica da Universidade de S\~ao Paulo (IME-USP),  \\ 
	Rua do Mat\~ao, 1010, S\~ao Paulo 05508-090,
	 SP, Brazil.\\
dflopezga@ime.usp.br}

\end{document}